\documentclass[journal]{IEEEtran}

\usepackage{mathrsfs}
\usepackage{amsmath}
\usepackage{amsfonts}
\usepackage{amssymb}
\usepackage{graphicx}
\usepackage{subfigure}
\usepackage{cite}

\newtheorem{thm}{Theorem}

\newtheorem{cor}{Corollary}
\newtheorem{rem}{Remark}
\newtheorem{lem}{Lemma}

\newtheorem{defi}{Definition}
\def\Ds{\displaystyle}
\def\diag{\mathrm{diag}}
\def\tp{\mathrm{T}}

\def\sgn{\mathrm{sgn}}

\renewcommand\baselinestretch{0.999}

\begin{document}

\title{Convergence Behavior Analysis of Directed Signed Networks Subject to Nonidentical Topologies}

\author{Deyuan Meng, {\it Senior Member}, {\it IEEE}, Mingjun Du, and Jianqiang Liang %<-this stops a space\\\vspace{1cm}September 15, 2007
\thanks{This work was supported by the National Natural Science Foundation of China under Grant 61922007 and Grant 61873013. {\it(Corresponding author: Deyuan Meng.)}}
\thanks{The authors are with the Seventh Research Division, Beihang University (BUAA), Beijing 100191, P. R. China, and also with the School of Automation Science and Electrical Engineering, Beihang University (BUAA), Beijing 100191, P. R. China (e-mail: dymeng@buaa.edu.cn).}
}

%\markboth{Journal of \LaTeX\ Class Files,~Vol.~6, No.~1, January~2007}%
%{Shell \MakeLowercase{\textit{et al.}}: Bare Demo of IEEEtran.cls for Journals}
%\IEEEpubid{0000--0000/00\$00.00~\copyright~2007 IEEE}
%\IEEEspecialpapernotice{(Invited Paper)}

\maketitle
\date{}

%%%%%%%%%%%%%%%%%%%%%%%%%%%%%%%%%%%%%%%%%%%%%%%%%%%%%%%%%%%%%%%%%%%%%%%%%%%%%%%%%
\begin{abstract}
This paper addresses the behavior analysis problems for directed signed networks that involve cooperative-antagonistic interactions among agents. Of particular interest is to explore the convergence behaviors of directed signed networks with agents of mixed first-order and second-order dynamics. Further, the agents are subject to nonidentical topologies represented by two different signed digraphs that have a strongly connected union. It is shown that when considering signed networks subject to sign-consistent nonidentical topologies, polarization (respectively, neutralization) can be achieved if and only if the union of two signed digraphs is structurally balanced (respectively, unbalanced). By comparison, signed networks can always be guaranteed to become neutralized in the presence of sign-inconsistent nonidentical topologies.
\end{abstract}
\begin{keywords}
Signed network, nonidentical topology, polarization, sign-consistency, structural balance.
\end{keywords}

%%%%%%%%%%%%%%%%%%%%%%%%%%%%%%%%%%%%%%%%%%%%%%%%%%%%%%%%%%%%%%%%%%%%%%%%%%%%%%%%%%%%%%%%%%%%%
\section{Introduction}\label{sec1}

Networks with multiple agents (vertices or nodes) that may interact cooperatively or antagonistically with each other have attracted considerable research interests recently, especially in the two areas of multi-agent systems and of complex systems. They also have potential values in many applications, such as social sciences, natural sciences, economics, and robotics (e.g., see \cite[Section 6]{pt:18} for more discussions). This class of networks subject to cooperative-antagonistic interactions is called {\it signed networks} for convenience since, to describe the communication topologies of them, signed graphs are generally employed such that the positive and negative weights of their arcs can be used to represent cooperative and antagonistic interactions between agents, respectively. Particularly, signed networks can include as a trivial case traditional networks, which are called {\it unsigned networks} for distinction, with the communication topologies of agents represented by the traditional unsigned (or nonnegative) graphs (see, e.g., \cite{rb:05,hhg:06,ycc:10,at:13,qzg:12,qy:13,lb:15,qzgmf:17,nmd:16,mndz:18}).

In general, unsigned networks enable the agents to achieve agreement (or consensus) since they only have the cooperative interactions among agents (see \cite{rb:05,hhg:06,ycc:10,at:13,qzg:12,qy:13,lb:15,qzgmf:17,nmd:16,mndz:18}). In comparison with unsigned networks, signed networks behave differently owing to the simultaneous existence of cooperations and antagonisms among agents. Various classes of dynamic behaviors emerge in signed networks, such as polarization or bipartite consensus \cite{a:13,vm:14,h:14,pmc:16,zc:17,zjy:17,m:18,wwyy:18}, sign consensus \cite{jzc:17}, modulus consensus \cite{msjch:16,lcbb:17,xhc:19}, interval bipartite consensus \cite{hz:14,mdj:16,xcj:16,dm:18}, bipartite flocking \cite{fzw:14} and bipartite containment tracking \cite{m:17}. It is worth noticing that polarization, instead of agreement, plays a fundamentally important role in analyzing the behaviors of signed networks. Moreover, neutralization and polarization are counterparts for signed networks, which correspond to the structural unbalance and balance of them, respectively. In particular, agreement can be viewed as a trivial case of polarization, in accordance with that the unsigned networks are included as a special case of the structurally balanced signed networks. Regardless of the great development on the behavior analysis of signed networks, most existing results make contributions to signed networks with the identical topologies that can be described with a single graph.

Different from identical topologies, nonidentical topologies generally need two different graphs to identify the communications among agents. In unsigned networks, it has already been realized that the agreement results of agents with nonidentical topologies are not simply an extension of those in the presence of an identical topology, but subject to challenging difficulties in both protocol design and agreement analysis (see, e.g., \cite{qzg:12,qy:13,lb:15,qzgmf:17,nmd:16,mndz:18}). New design and analysis approaches are usually required to be explored for unsigned networks owing to the presence of nonidentical topologies. Take, for instance, \cite{qzg:12,qy:13,lb:15,qzgmf:17} that resort to new Lyapunov analysis approaches because they generally need to treat two different classes of convergence problems on second-order unsigned networks with nonidentical topologies \cite{qzg:12,qy:13,lb:15,qzgmf:17}. In \cite{nmd:16,mndz:18}, a new dynamic graph approach based on the quadratic matrix polynomials has been introduced to cope with the ill effects of nonidentical topologies on the agreement analysis of unsigned networks. How to implement the behavior analysis of directed signed networks with mixed first-order and second-order dynamics, and what are the challenging problems of them caused by nonidentical topologies? In addition, what are the main differences between the convergence analyses of signed and unsigned networks against nonidentical topologies, and further are they fundamental? Even though some attempts have been made to answer these questions in, e.g., \cite{m:18,lm:18}, the frequency-domain analysis approaches have been utilized, with which either the effects of nonidentical topologies have not been considered \cite{m:18}, or the strict symmetry of information communications among agents has been imposed \cite{lm:18}. Also, a new class of sign-consistency problems due to nonidentical topologies has been discovered for signed networks in \cite{lm:18}, which however is only discussed for undirected networks, and it is also left to develop the behaviors of signed networks when the sign-consistency condition is violated.

In this paper, we aim at directed signed networks subject to nonidentical topologies and analyze the convergence behaviors of them though the agents have mixed first-order and second-order dynamics. It is shown that a sign-consistency problem is caused for sign patterns of nonidentical topologies and found to play a dominant role in investigating convergence behaviors of signed networks. The sign-consistency is a distinct problem resulting from the nonidentical topologies of signed networks, which naturally disappears for the cases of identical topologies or of unsigned networks. When two signed digraphs (short for directed graphs) representing the nonidentical communication topologies of signed networks are sign-consistent, polarization emerges for the agents if and only if the union of two signed digraphs is structurally balanced; and otherwise, neutralization arises. In the presence of sign-inconsistent nonidentical topologies, these results do not work any longer, and correspondingly the Lyapunov analysis is not sufficient to address convergence problems for signed networks. To overcome challenges arising from the sign-inconsistency, an $M$-matrix method is proposed, which reveals that signed networks always become neutralized.

{\it Notations:} Throughout this paper, let $1_{n}=[1,1,\cdots,1]^{\tp}\in\mathbb{R}^{n}$, and $I$ and $0$ be the identity and null matrices with compatible dimensions, respectively. Denote $\mathcal{I}_{n}=\{1,2,\cdots,n\}$, and two sets of matrices as
\[\aligned
\mathcal{D}_{n}&=\left\{D=\diag\left\{d_{1},d_{2},\cdots,d_{n}\right\}:d_{i}\in\{\pm1\},\forall i\in\mathcal{I}_{n}\right\}\\ \mathcal{Z}_{n}&=\left\{Z=\left[z_{ij}\right]\in\mathbb{R}^{n\times n}:z_{ij}\leq0,\forall i\neq j,\forall i,j\in\mathcal{I}_{n}\right\}
\endaligned
\]

\noindent where $\diag\{d_{1},d_{2},\cdots,d_{n}\}$ is a diagonal matrix with diagonal entries given in order as $d_{1}$, $d_{2}$, $\cdots$, $d_{n}$. For any $A\in\mathbb{R}^{n\times m}$, we call it a nonnegative matrix, denoted by $A\geq0$, if all entries of $A$ are nonnegative. When $m=n$ (i.e., $A$ is square), $\det(A)$ is the determinant of $A$, and $\lambda_{i}(A)$, $\forall i\in\mathcal{I}_{n}$ is an eigenvalue of $A$. If all eigenvalues of $A$ are real, then $\lambda_{\max}(A)$ and $\lambda_{\min}(A)$ are the largest and smallest eigenvalues of $A$, respectively. Besides, for $A=[a_{ij}]\in\mathbb{R}^{n\times m}$, we denote
\[\aligned
\Delta_{A}&=\diag\left\{\sum_{j=1}^{m}a_{1j},\sum_{j=1}^{m}a_{2j},\cdots,\sum_{j=1}^{m}a_{nj}\right\},
&\left|A\right|&=\left[\left|a_{ij}\right|\right]\\
A^{+}&=\left[a_{ij}^{+}\right],
&A^{-}&=\left[a_{ij}^{-}\right]
\endaligned
\]

\noindent where $\left|a_{ij}\right|=\sgn\left(a_{ij}\right)a_{ij}$, together with $\sgn\left(a_{ij}\right)$ taking the sign value of $a_{ij}$; $a_{ij}^{+}=a_{ij}$ if $a_{ij}>0$ and $a_{ij}^{+}=0$, otherwise; and $a_{ij}^{-}=a_{ij}$ if $a_{ij}<0$ and $a_{ij}^{-}=0$, otherwise. Any square matrix $A\in\mathbb{R}^{n\times n}$ is said to be Hurwitz stable (respectively, positive stable) if, for every $i\in\mathcal{I}_{n}$, $\lambda_{i}(A)$ has a negative (respectively, positive) real part. By \cite[Definition 2.5.2]{hj:91}, we say that $A\in\mathcal{Z}_{n}$ is an $M$-matrix if $A$ is positive stable.

%%%%%%%%%%%%%%%%%%%%%%%%%%%%%%%%%%%%%%%%%%%%%%%%%%%%%%%%%%%%%%%%%%%%%%%%%%%%%%%%%%%%%%%%%%%%%
\section{Signed Digraphs and Associated Networks}\label{sec2}

\subsection{Signed Digraphs}

A digraph is a pair $\mathscr{G}=\left(\mathcal{V},\mathcal{E}\right)$ that consists of a vertex set $\mathcal{V}=\left\{v_{i}:\forall i\in\mathcal{I}_{n}\right\}$ and an arc set $\mathcal{E}\subseteq\left\{\left(v_{j},v_{i}\right):\forall j\neq i\right\}$. If there exists an arc $\left(v_{j},v_{i}\right)\in\mathcal{E}$, $\forall j\neq i$, then $v_{i}$ has a ``neighbor $v_{j}$,'' with which the index set $\mathcal{N}_{i}=\left\{j:(v_{j},v_{i})\in\mathcal{E}\right\}$ denotes all the neighbors of $v_{i}$. If there exist sequential arcs $(v_{i},v_{l_{1}})$, $(v_{l_{1}},v_{l_{2}})$, $\cdots$, $(v_{l_{m-1}},v_{j})$ formed by distinct vertices $v_{i}$, $v_{l_{1}}$, $\cdots$, $v_{l_{m-1}}$, $v_{j}$, then $\mathscr{G}$ admits a path from $v_{i}$ to $v_{j}$. We say that $\mathscr{G}$ is strongly connected if, for each pair of distinct vertices, there exists a path between them. For the digraph $\mathscr{G}$, let two other digraphs $\mathscr{G}^{c}=(\mathcal{V},\mathcal{E}^{c})$ and  $\mathscr{G}^{d}=(\mathcal{V},\mathcal{E}^{d})$ possess the same vertex set $\mathcal{V}$ as $\mathscr{G}$. If $\mathcal{E}=\mathcal{E}^{c}\cup\mathcal{E}^{d}$ holds, then $\mathscr{G}$ is called the union of $\mathscr{G}^{c}$ and $\mathscr{G}^{d}$, which is denoted by $\mathscr{G}=\mathscr{G}^{c}\cup\mathscr{G}^{d}$.

A digraph $\mathscr{G}$ is called a signed digraph if it is associated with a real adjacency matrix $B=\left[b_{ij}\right]\in\mathbb{R}^{n\times n}$, where $b_{ij}\neq0\Leftrightarrow\left(v_{j},v_{i}\right)\in\mathcal{E}$ (and in particular, $b_{ii}=0$, $\forall i\in\mathcal{I}_{n}$ due to $\left(v_{i},v_{i}\right)\not\in\mathcal{E}$). Let the signed digraph $\mathscr{G}$ associated with $B$ be denoted as $\mathscr{G}\left(B\right)=\left(\mathcal{V},\mathcal{E},B\right)$. The Laplacian matrix of $\mathscr{G}\left(B\right)$ is defined as $L_{B}=\left[l^{B}_{ij}\right]\in\mathbb{R}^{n\times n}$ whose elements satisfy $l^{B}_{ij}=\sum_{k\in\mathcal{N}_{i}}\left|b_{ik}\right|$ if $j=i$ and $l^{B}_{ij}=-b_{ij}$, otherwise. By following \cite{a:13}, $\mathscr{G}\left(B\right)$ is said to be structurally balanced if $\mathcal{V}$ admits a bipartition $\{\mathcal{V}^{(1)},\mathcal{V}^{(2)}:\mathcal{V}^{(1)}\cup\mathcal{V}^{(2)}=\mathcal{V},\mathcal{V}^{(1)}\cap\mathcal{V}^{(2)}=\O\}$ such that $b_{ij}\geq0$, $\forall v_{i}, v_{j}\in\mathcal{V}^{(l)}$ for $l\in\{1, 2\}$ and $b_{ij}\leq0$, $\forall v_{i}\in\mathcal{V}^{(l)}$, $\forall v_{j}\in\mathcal{V}^{(q)}$ for $l\neq q$, $\forall l, q\in\{1, 2\}$; and it is said to be structurally unbalanced, otherwise. Correspondingly, $\mathscr{G}\left(B\right)$ is structurally balanced (respectively, unbalanced) if and only if there exists some (respectively, does not exist any) $D\in\mathcal{D}_{n}$ such that $DBD=\left|B\right|$ (see also \cite{a:13}). In particular, when $B\geq0$, $\mathscr{G}\left(B\right)$ collapses into a traditional unsigned digraph that is a trivial case of structurally balanced signed digraphs.

For how to reasonably represent the union $\mathscr{G}\left(B^{c}\right)\cup\mathscr{G}\left(B^{d}\right)$ of two signed digraphs $\mathscr{G}\left(B^{c}\right)=\left(\mathcal{V},\mathcal{E}^{c},B^{c}\right)$ and $\mathscr{G}\left(B^{d}\right)=\left(\mathcal{V},\mathcal{E}^{d},B^{d}\right)$, we generally only have its vertex set $\mathcal{V}$ and arc set $\mathcal{E}^{c}\cup\mathcal{E}^{d}$, but how to associate it with an appropriate adjacency matrix is an open problem. To handle this problem, we present a sign-consistency property for any two signed digraphs in the following notion.

\begin{defi}\label{def1}
For any two matrices $B^{c}=\left[b_{ij}^{c}\right]\in\mathbb{R}^{n\times n}$ and $B^{d}=\left[b_{ij}^{d}\right]\in\mathbb{R}^{n\times n}$, if $b_{ij}^{c}b_{ij}^{d}\geq0$, $\forall i$, $j\in\mathcal{I}_{n}$, then $\mathscr{G}\left(B^{c}\right)$ and $\mathscr{G}\left(B^{d}\right)$ are called {\it sign-consistent signed digraphs}; and otherwise, they are called {\it sign-inconsistent signed digraphs}.
\end{defi}

As a benefit of Definition \ref{def1}, we provide a definition for the adjacency matrix of the union of any two signed digraphs.

\begin{defi}\label{def2}
For any two sign-consistent signed digraphs $\mathscr{G}\left(B^{c}\right)$ and $\mathscr{G}\left(B^{d}\right)$, their union $\mathscr{G}\left(B^{c}\right)\cup\mathscr{G}\left(B^{d}\right)$ can be defined by $\mathscr{G}\left(B\right)$ if the adjacency matrix $B$ is given by
\begin{equation*}
B=\alpha B^{c}+\beta B^{d}
\end{equation*}

\noindent where $\alpha\in\mathbb{R}$ and $\beta\in\mathbb{R}$ are any scalars such that $\alpha\beta>0$.
\end{defi}

\begin{rem}\label{rem7}
If $\mathscr{G}\left(B^{c}\right)$ and $\mathscr{G}\left(B^{d}\right)$ are sign-consistent, then $b_{ij}^{c}b_{ij}^{d}\geq0$ holds for all $i$, $j\in\mathcal{I}_{n}$ by Definition \ref{def1}. Hence, we apply $\alpha\beta>0$ and can verify
\[
\alpha b_{ij}^{c}+\beta b_{ij}^{d}=0
~\Leftrightarrow~b_{ij}^{c}=0~\hbox{and}~b_{ij}^{d}=0,~~~\forall i,j\in\mathcal{I}_{n}.
\]

\noindent Conversely, $\alpha b_{ij}^{c}+\beta b_{ij}^{d}\neq0$ if and only if $b_{ij}^{c}\neq0$ or $b_{ij}^{d}\neq0$. This yields $\mathscr{G}\left(B\right)=\mathscr{G}\left(\alpha B^{c}+\beta B^{d}\right)=\mathscr{G}\left(B^{c}\right)\cup\mathscr{G}\left(B^{d}\right)$, as adopted in Definition \ref{def2}. By contrast, when $\mathscr{G}\left(B^{c}\right)$ and $\mathscr{G}\left(B^{d}\right)$ are sign-inconsistent, $b_{ij}^{c}b_{ij}^{d}\geq0$ may not hold for any $i$, $j\in\mathcal{I}_{n}$, and then
\[\aligned
\alpha b_{ij}^{c}+\beta b_{ij}^{d}\neq0
&~\Rightarrow~b_{ij}^{c}\neq0~\hbox{or}~b_{ij}^{d}\neq0,~~~\forall i,j\in\mathcal{I}_{n}
\endaligned
\]

\noindent but the opposite may not be true. Namely, $\mathscr{G}\left(\alpha B^{c}+\beta B^{d}\right)
\subseteq\mathscr{G}\left(B^{c}\right)\cup\mathscr{G}\left(B^{d}\right)$ can only be derived. However, if we select suitable $\alpha$ and $\beta$ such that
\begin{equation}\label{eq02}
b_{ij}^{c}\neq0~\hbox{or}~b_{ij}^{d}\neq0
~\Rightarrow~\alpha b_{ij}^{c}+\beta b_{ij}^{d}\neq0,~~~\forall i,j\in\mathcal{I}_{n}
\end{equation}

\noindent then we also have $\mathscr{G}\left(B^{c}\right)\cup\mathscr{G}\left(B^{d}\right)
\subseteq\mathscr{G}\left(\alpha B^{c}+\beta B^{d}\right)$. As a consequence, only suitable selections of $\alpha$ and $\beta$ satisfying (\ref{eq02}) can ensure that $\mathscr{G}\left(\alpha B^{c}+\beta B^{d}\right)$ is qualified as the union of any sign-inconsistent $\mathscr{G}\left(B^{c}\right)$ and $\mathscr{G}\left(B^{d}\right)$.
\end{rem}

As revealed by Definition \ref{def2} and Remark \ref{rem7}, the sign patterns have great effect on properties of nonidentical signed digraphs. It is distinct for nonidentical signed digraphs, which disappears for identical signed digraphs or unsigned digraph pairs. The sign-consistency property of Definition \ref{def1} can also be extended to arbitrarily finite number of signed digraphs.

\subsection{Signed Networks}

Consider two signed digraphs $\mathscr{G}\left(B^{c}\right)$ and $\mathscr{G}\left(B^{d}\right)$ for $B^{c}=\left[b_{ij}^{c}\right]\in\mathbb{R}^{n\times n}$ and $B^{d}=\left[b_{ij}^{d}\right]\in\mathbb{R}^{n\times n}$, respectively. For the sake of distinguishing the neighbor index sets of agent $v_{i}$ in $\mathscr{G}\left(B^{c}\right)$ and $\mathscr{G}\left(B^{d}\right)$, we represent them as $\mathcal{N}_{i}^{c}$ and $\mathcal{N}_{i}^{d}$, respectively. Let $x_{i}(t)\in\mathbb{R}$ and $y_{i}(t)\in\mathbb{R}$ denote the states of $v_{i}$, and then for each $i\in\mathcal{I}_{n}$, the dynamics of $v_{i}$ fulfill
\begin{equation}\label{eq01}
\left\{\aligned
\dot{x}_{i}(t)&=\sum_{j\in\mathcal{N}_{i}^{c}}b_{ij}^{c}\left[x_{j}(t)-\sgn\left(b_{ij}^{c}\right)x_{i}(t)\right]+y_{i}(t)\\
\dot{y}_{i}(t)&=-k\left\{y_{i}(t)
+\sum_{j\in\mathcal{N}_{i}^{d}}b_{ij}^{d}\left[x_{j}(t)-\sgn\left(b_{ij}^{d}\right)x_{i}(t)\right]\right\}
\endaligned\right.
\end{equation}

\noindent where $k>0$ is a damping rate. It is worth highlighting that (\ref{eq01}) is subject to nonidentical topologies represented by two signed digraphs $\mathscr{G}\left(B^{c}\right)$ and $\mathscr{G}\left(B^{d}\right)$. The nonidentical topologies lead to that (\ref{eq01}) can represent signed networks of agents with mixed first-order and second-order dynamics. Two extreme cases of this mixed-order representation are given below.
\begin{enumerate}
\item[i)]
Let $\mathcal{I}_{n}^{1}=\left\{i:b_{ij}^{d}=0,\forall j\in\mathcal{I}_{n}\right\}$, and then for each $i\in\mathcal{I}_{n}^{1}$, (\ref{eq01}) collapses into describing that $v_{i}$ has the single-integrator dynamics.

\item[ii)]
Let $\mathcal{I}_{n}^{2}=\left\{i:b_{ij}^{c}=0,\forall j\in\mathcal{I}_{n}\right\}$, and then for each $i\in\mathcal{I}_{n}^{2}$, (\ref{eq01}) becomes the description of $v_{i}$ with the double-integrator dynamics.
\end{enumerate}

\noindent When $\mathcal{I}_{n}^{1}\cup\mathcal{I}_{n}^{2}=\mathcal{I}_{n}$, (\ref{eq01}) represents a signed network subject to mixed first-order and second-order dynamics, which only has been studied for the trivial case without involving antagonistic interactions among agents (see, e.g., \cite{fxlz:15,whwcp:16,zzw:11}).

\begin{rem}\label{rem1}
In particular, when $\mathcal{I}_{n}^{1}=\mathcal{I}_{n}$ holds, (\ref{eq01}) essentially describes a signed network with single-integrator dynamics associated with $\mathscr{G}\left(B^{c}\right)$ (see, e.g., \cite{a:13,hz:14,m:17,dm:18,mdj:16}). If $\mathcal{I}_{n}^{2}=\mathcal{I}_{n}$ holds, then (\ref{eq01}) collapses into a double-integrator signed network associated with $\mathscr{G}\left(B^{d}\right)$ (see, e.g., \cite{fzw:14,m:18}). Hence, we may build a relationship between single-integrator and double-integrator signed networks through the study on (\ref{eq01}), where the effects of nonidentical topologies on behaviors of directed signed networks may emerge as a crucial issue.
\end{rem}

For arbitrary initial states $x_{i}(0)$ and $y_{i}(0)$, $\forall i\in\mathcal{I}_{n}$, we say that the system (\ref{eq01}) associated with the nonidentical topologies $\mathscr{G}\left(B^{c}\right)$ and $\mathscr{G}\left(B^{d}\right)$ achieves
\begin{enumerate}
\item polarization (or bipartite consensus) if
\begin{equation*}
\lim_{t\to\infty}\left|x_{i}(t)\right|=\theta~\hbox{and}~\lim_{t\to\infty}y_{i}(t)=0,~\forall i\in\mathcal{I}_{n}
\end{equation*}

\item neutralization (or asymptotic stability) if
\begin{equation*}
\lim_{t\to\infty}x_{i}(t)=0~\hbox{and}~\lim_{t\to\infty}y_{i}(t)=0,~\forall i\in\mathcal{I}_{n}
\end{equation*}
\end{enumerate}

\noindent where $\theta\geq0$ denotes some constant depending on $x_i(0)$ and $y_i(0)$, $\forall i\in \mathcal{I}_n$.

%%%%%%%%%%%%%%%%%%%%%%%%%%%%%%%%%%%%%%%%%%%%%%%%%%%%%%%%%%%%%%%%%%%%%%%%%%%%%%%%%%%%%%%%%%%%%
\section{Behavior Analysis Results}\label{sec3}

To characterize the dynamic behaviors of the directed signed network represented by (\ref{eq01}), let us denote $x(t)=[x_{1}(t)$, $x_{2}(t)$, $\cdots$, $x_{n}(t)]^{\tp}$ and $y(t)=[y_{1}(t)$, $y_{2}(t)$, $\cdots$, $y_{n}(t)]^{\tp}$. Then we can express (\ref{eq01}) in a compact vector form of
\begin{equation}\label{eq03}
\begin{bmatrix}
\dot{x}(t)\\
\dot{y}(t)
\end{bmatrix}
=\begin{bmatrix}
-L_{B^{c}}&I\\
-kL_{B^{d}}&-kI
\end{bmatrix}
\begin{bmatrix}
x(t)\\
y(t)
\end{bmatrix}.
\end{equation}

\noindent For the trivial case when $B^{c}=0$, the convergence properties of (\ref{eq03}) are exploited in \cite[Theorem 4]{m:18}, where the selection of $k$ depends heavily on the topology of signed networks described by $\mathscr{G}\left(B^{d}\right)$. It is, however, obvious that the behavior analysis and result of \cite[Theorem 4]{m:18} do not work for signed networks described by (\ref{eq03}) subject to nonidentical $\mathscr{G}\left(B^{c}\right)$ and $\mathscr{G}\left(B^{d}\right)$. Further, how to select the rate $k$ to overcome the effects of nonidentical topologies on the behavior analysis of signed networks may become much more challenging.

To proceed to address the aforementioned issues, we denote $\mathscr{G}\left(B^{c}\right)\cup\mathscr{G}\left(B^{d}\right)=\mathscr{G}\left(B\right)$ and $X(t)=\left[x^{\tp}(t),y^{\tp}(t)\right]^{\tp}$, and present a nonsingular linear transformation on the system (\ref{eq03}).

\begin{lem}\label{lem1}
If a nonsingular linear transformation is used as
\[
Y(t)=\Theta X(t)~\hbox{with}~\Theta=\begin{bmatrix}kI&I\\I&0\end{bmatrix}
\]

\noindent then the system (\ref{eq03}) can be equivalently transformed into
\begin{equation}\label{eq04}
\dot{Y}(t)=
\underbrace{
\begin{bmatrix}
0&-k\left(L_{B^{c}}+L_{B^{d}}\right)\\
I&-\left(kI+L_{B^{c}}\right)
\end{bmatrix}}_{\Ds\triangleq\Gamma}
Y(t).
\end{equation}
\end{lem}

\begin{proof}
With $Y(t)=\Theta X(t)$, (\ref{eq04}) is direct from (\ref{eq03}).
\end{proof}

For the system (\ref{eq04}), there are two classes of nontrivial block matrices in its system matrix, for which we have properties in the lemma below.

\begin{lem}\label{lem2}
For any positive scalars $k>0$, $\alpha>0$ and $\beta>0$ and any $n$-by-$n$ real matrices $B^{c}$ and $B^{d}$, the following hold:
\begin{enumerate}
\item[1)]
$kI+L_{B^{c}}$ is positive stable;

\item[2)]
$\alpha L_{B^{c}}+\beta L_{B^{d}}=L_{\alpha B^{c}+\beta B^{d}}$ if and only if two signed digraphs $\mathscr{G}\left(B^{c}\right)$ and $\mathscr{G}\left(B^{d}\right)$ are sign-consistent.
\end{enumerate}
\end{lem}

\begin{proof}
This lemma is immediate by leveraging definitions of Laplacian matrices and sign-consistency in Definition \ref{def1}.
\end{proof}

\begin{rem}\label{rem3}
From Lemma \ref{lem1}, it is clear that the union $\mathscr{G}\left(B\right)$ of signed digraphs $\mathscr{G}\left(B^{c}\right)$ and $\mathscr{G}\left(B^{d}\right)$ may play an important role in the dynamic behaviors of signed networks described by (\ref{eq03}). As a benefit of Lemma \ref{lem2}, a candidate of $\mathscr{G}\left(B\right)$ is given by $B=\alpha B^{c}+\beta B^{d}$ for any $\alpha>0$ and $\beta>0$ in the presence of any sign-consistent signed digraphs $\mathscr{G}\left(B^{c}\right)$ and $\mathscr{G}\left(B^{d}\right)$. In particular, it follows that $L_{B^{c}}+\delta L_{B^{d}}=L_{B^{c}+\delta B^{d}}$, $\forall\delta>0$ is the Laplacian matrix of $\mathscr{G}\left(B\right)$ for $B=B^{c}+\delta B^{d}$. However, these properties do not work any longer for sign-inconsistent signed digraphs $\mathscr{G}\left(B^{c}\right)$ and $\mathscr{G}\left(B^{d}\right)$ by noting Remark \ref{rem7}.
\end{rem}

\subsection{Sign-Consistent Nonidentical Topologies}

When considering sign-consistent nonidentical topologies of signed networks described by (\ref{eq03}), their convergence behaviors are tied closely with the structural balance of signed digraphs. To reveal this property, we give a helpful lemma related to the structural balance property.

\begin{lem}\label{lem3}
Let the union $\mathscr{G}\left(B\right)$ of any two sign-consistent signed digraphs $\mathscr{G}\left(B^{c}\right)$ and $\mathscr{G}\left(B^{d}\right)$ be strongly connected. Then for $B=B^{c}+B^{d}$, the following two results hold.
\begin{enumerate}
\item[1)]
If $\mathscr{G}\left(B\right)$ is structurally balanced, there exists a unique positive definite matrix $W\in\mathbb{R}^{(n-1)\times(n-1)}$ such that
\begin{equation}\label{Eq8}
\aligned
&\left[ED\left(L_{B^{c}}+L_{B^{d}}\right)DF\right]^{\tp}W\\
&~~~~~~~~~~+W\left[ED\left(L_{B^{c}}+L_{B^{d}}\right)DF\right]=I
\endaligned
\end{equation}

\noindent where $D\in\mathcal{D}_{n}$ is such that $DL_{B}D=L_{\left|B\right|}$, and $E$ and $F$ are matrices given by
\[E=\begin{bmatrix}-1_{n-1}&I\end{bmatrix}\in\mathbb{R}^{(n-1)\times n},
~~F=\begin{bmatrix}0\\I\end{bmatrix}\in\mathbb{R}^{n\times(n-1)}.
\]

\item[2)]
If $\mathscr{G}\left(B\right)$ is structurally unbalanced, there exists a unique positive definite matrix $H\in\mathbb{R}^{n\times n}$ such that
\begin{equation}\label{Eq22}
\left(L_{B^{c}}+L_{B^{d}}\right)^{\tp}H
+H\left(L_{B^{c}}+L_{B^{d}}\right)=I.
\end{equation}
\end{enumerate}
\end{lem}

\begin{proof}
With Lemma \ref{lem2} and \cite[Theorems 4.1 and 4.2]{dm:18}, this lemma follows from the Lyapunov stability theory.
\end{proof}

As a consequence of Lemma \ref{lem2}, we can accomplish a further result of Lemma \ref{lem3} under the same connectivity and structural balance conditions.

\begin{cor}\label{cor1}
If the union $\mathscr{G}\left(B\right)$ of any two sign-consistent signed digraphs $\mathscr{G}\left(B^{c}\right)$ and $\mathscr{G}\left(B^{d}\right)$ is strongly connected, then given any $\alpha>0$ and $\beta>0$, $ED\left(\alpha L_{B^{c}}+\beta L_{B^{d}}\right)DF$ (respectively, $\alpha L_{B^{c}}+\beta L_{B^{d}}$) is positive stable provided that $\mathscr{G}\left(B\right)$ is structurally balanced (respectively, structurally unbalanced) for $B=\alpha B^{c}+\beta B^{d}$.
\end{cor}

By Lemma \ref{lem3} and Corollary \ref{cor1}, we consider $B=B^{c}+B^{d}$ and introduce an index with respect to any $\delta>1$ as
\begin{equation}\label{Eq33}
\mu=\frac{\Ds\lambda_{\max}\left(\Phi\right)\lambda_{\max}\left(\Psi\Psi^{\tp}\right)}{\Ds2(\delta-1)}
\end{equation}

\noindent where $\Phi$ and $\Psi$, respectively, satisfy
\begin{equation*}\label{}
\Phi=\left\{\aligned
&W,&&\hbox{if $\mathscr{G}\left(B\right)$ is structurally balanced;}\\
&H,&&\hbox{if $\mathscr{G}\left(B\right)$ is structurally unbalanced;}
\endaligned\right.
\end{equation*}

\noindent and
\begin{equation*}\label{}
\Psi=\left\{\aligned
&ED\left(L_{B^{c}}+\delta L_{B^{d}}\right)DF,\\
&~~~~~~~~~~~\hbox{if $\mathscr{G}\left(B\right)$ is structurally balanced;}\\
&L_{B^{c}}+\delta L_{B^{d}},\\
&~~~~~~~~~~~\hbox{if $\mathscr{G}\left(B\right)$ is structurally unbalanced.}
\endaligned\right.
\end{equation*}

\noindent Note that $\mathscr{G}\left(B\right)$ with $B=B^{c}+\delta B^{d}$ has the same connectivity and structural balance properties for any $\delta>0$. This, together with Corollary \ref{cor1}, ensures that in (\ref{Eq33}), $\lambda_{\max}\left(\Psi\Psi^{\tp}\right)>0$ under the strong connectivity of $\mathscr{G}\left(B\right)$. Consequently, $\mu>0$ can be guaranteed from the definition (\ref{Eq33}).

With the abovementioned discussions, a convergence result tied to the structural balance of signed digraphs is established for directed signed networks in the presence of sign-consistent nonidentical topologies.

\begin{thm}\label{thm1}
Consider the system (\ref{eq03}) described by any sign-consistent signed digraphs $\mathscr{G}\left(B^{c}\right)$ and $\mathscr{G}\left(B^{d}\right)$, and let $k>\mu$ be selected for any $\delta>1$. If the union $\mathscr{G}\left(B\right)$ of $\mathscr{G}\left(B^{c}\right)$ and $\mathscr{G}\left(B^{d}\right)$ is strongly connected, then under any initial conditions of $x(0)\in\mathbb{R}^{n}$ and of $y(0)\in\mathbb{R}^{n}$, the following results hold.
\begin{enumerate}
\item[1)]
Polarization can be achieved for the system (\ref{eq03}) if and only if $\mathscr{G}\left(B\right)$ is structurally balanced for $B=B^{c}+B^{d}$. Moreover, the converged solution of (\ref{eq03}) is given by
\[
\aligned
\lim_{t\to\infty}x(t)
&=\left\{\nu^{\tp}D\left[x(0)+k^{-1}y(0)\right]\right\}D1_{n}\\
\lim_{t\to\infty}y(t)
&=0
\endaligned
\]

\noindent for some $D\in\mathcal{D}_{n}$ satisfying $DL_{B}D=L_{\left|B\right|}$ and some $\nu\in\mathbb{R}^{n}$ satisfying $\nu^{\tp}\left(DL_{B}D\right)=0$ and $\nu^{\tp}1_{n}=1$.

\item[2)]
Neutralization can be achieved for the system (\ref{eq03}) such that $\lim_{t\to\infty}x(t)=0$ and $\lim_{t\to\infty}y(t)=0$ if and only if $\mathscr{G}\left(B\right)$ is structurally unbalanced for $B=B^{c}+B^{d}$.
\end{enumerate}
\end{thm}

\begin{proof}
We develop a Lyapunov-based analysis method to prove this theorem. For the details, see the Appendix A.
\end{proof}

\begin{rem}\label{rem4}
In Theorem \ref{thm1}, the convergence behaviors are analyzed essentially for signed networks subject to nonidentical directed topologies. Like the conventional behavior results for first-order signed networks exploited in, e.g., \cite{a:13,pmc:16,mdj:16,dm:18}, polarization and neutralization are developed in Theorem \ref{thm1} and closely tied to the structural balance and unbalance of signed digraphs, respectively. However, Theorem \ref{thm1} encounters dealing with the effect of nonidentical topologies on behaviors of signed networks, which is addressed by exploring the sign-consistence property of signed digraphs. This actually is a new challenging problem that has not been considered for directed signed networks in the literature.
\end{rem}

\subsection{Sign-Inconsistent Nonidentical Topologies}

When considering any sign-inconsistent signed digraphs, the union of them is always rendered with the structural unbalance. Another effect resulting from sign-inconsistence is that a linear combination of the Laplacian matrices of signed digraphs may not be employed as the Laplacian matrix of their union. Hence, the Lyapunov-based convergence analysis method in Theorem \ref{thm1} may not be directly implemented for signed networks subject to sign-inconsistent nonidentical topologies. To overcome this issue, we develop helpful properties of sign-inconsistency by exploring the properties of $M$-matrices based on the separation of cooperative and antagonistic interactions.

Let us revisit the algebraic equivalent transformation (\ref{eq04}) of the system (\ref{eq03}). However, it follows immediately from Lemma \ref{lem2} that $L_{B^{c}}+L_{B^{d}}=L_{B^{c}+B^{d}}$ may not hold any longer owing to the sign-inconsistence of any signed digraphs $\mathscr{G}\left(B^{c}\right)$ and $\mathscr{G}\left(B^{d}\right)$. This results in that we can not establish the stability properties of $L_{B^{c}}+L_{B^{d}}$ by directly resorting to the signed digraph $\mathscr{G}\left(B^{c}+B^{d}\right)$. With such an observation, we consider $B^{c}=B^{c+}+B^{c-}$ and $B^{d}=B^{d+}+B^{d-}$ for $B^{c+}\geq0$, $B^{d+}\geq0$, $B^{c-}\leq0$ and $B^{d-}\leq0$, and can obtain
\begin{equation}\label{Eq27}
\aligned
L_{B^{c}}+L_{B^{d}}
&=\left(L_{B^{c+}}+L_{B^{d+}}\right)+\left(\Delta_{\left|B^{c-}\right|}+\Delta_{\left|B^{d-}\right|}\right)\\
&~~~-\left(B^{c-}+B^{d-}\right)\\
&=\left(L_{B^{c+}+B^{d+}}
+\Delta_{\left|B^{c-}+B^{d-}\right|}\right)-\left(B^{c-}+B^{d-}\right)
\endaligned
\end{equation}

\noindent for which we can verify
\begin{equation}\label{Eq43}
L_{B^{c+}+B^{d+}}+\Delta_{\left|B^{c-}+B^{d-}\right|}\in\mathcal{Z}_{n}\\
\end{equation}

\noindent and
\begin{equation}\label{Eq28}
-\left(B^{c-}+B^{d-}\right)\geq0.
\end{equation}

\noindent Further, an $M$-matrix property of $L_{B^{c+}+B^{d+}}+\Delta_{\left|B^{c-}+B^{d-}\right|}\in\mathcal{Z}_{n}$ in (\ref{Eq43}) can be established in the following lemma.

\begin{lem}\label{lem6}
Consider any sign-inconsistent signed digraphs $\mathscr{G}\left(B^{c}\right)$ and $\mathscr{G}\left(B^{d}\right)$. If the union of $\mathscr{G}\left(B^{c}\right)$ and $\mathscr{G}\left(B^{d}\right)$ is strongly connected, then $L_{B^{c+}+B^{d+}}+\Delta_{\left|B^{c-}+B^{d-}\right|}$ is ensured to be an $M$-matrix.
\end{lem}

\begin{proof}
See the Appendix B.
\end{proof}

Based on Lemma \ref{lem6} and by taking advantage of the equivalence between the Statements 2.5.3.1 and 2.5.3.17 in \cite[Theorem 2.5.3]{hj:91}, we can derive that the $M$-matrix $L_{B^{c+}+B^{d+}}+\Delta_{\left|B^{c-}+B^{d-}\right|}$ is nonsingular and has a nonnegative inverse matrix $\left(L_{B^{c+}+B^{d+}}+\Delta_{\left|B^{c-}+B^{d-}\right|}\right)^{-1}\geq0$. If we insert these facts into (\ref{Eq27}), then we can deduce
\begin{equation}\label{Eq42}
\aligned
L_{B^{c}}&+L_{B^{d}}
=\left(L_{B^{c+}+B^{d+}}
+\Delta_{\left|B^{c-}+B^{d-}\right|}\right)\\
&\times\left[I-\left(L_{B^{c+}+B^{d+}}
+\Delta_{\left|B^{c-}+B^{d-}\right|}\right)^{-1}\left(B^{c-}+B^{d-}\right)\right]
\endaligned
\end{equation}

\noindent which is the product of an $M$-matrix and a nonnegative matrix. To proceed further with exploring this fact, we present a matrix stability property with respect to the sign-inconsistency.

\begin{lem}\label{lem5}
Consider any sign-inconsistent signed digraphs $\mathscr{G}\left(B^{c}\right)$ and $\mathscr{G}\left(B^{d}\right)$. If the union of $\mathscr{G}\left(B^{c}\right)$ and $\mathscr{G}\left(B^{d}\right)$ is strongly connected, then $L_{B^{c}}+L_{B^{d}}$ is positive stable. Further, there exists a unique positive definite matrix $H\in\mathbb{R}^{n\times n}$ such that (\ref{Eq22}) holds.
\end{lem}

\begin{proof}
See the Appendix B.
\end{proof}

By following the same steps as adopted in the derivation of Lemma \ref{lem5}, we can present a more general result based on the sign-inconsistency property of signed digraphs.

\begin{cor}\label{cor2}
For any $\alpha>0$ and $\beta>0$, if the union of any two sign-inconsistent signed digraphs $\mathscr{G}\left(B^{c}\right)$ and $\mathscr{G}\left(B^{d}\right)$ is strongly connected, then $\alpha L_{B^{c}}+\beta L_{B^{d}}$ is positive stable.
\end{cor}

Based on Lemma \ref{lem5} and Corollary \ref{cor2}, we consider any $\delta>1$ to define an index with respect to any sign-inconsistent signed digraphs $\mathscr{G}\left(B^{c}\right)$ and $\mathscr{G}\left(B^{d}\right)$ as
\begin{equation}\label{Eq41}
\zeta=\frac{\Ds\lambda_{\max}\left(H\right)
\lambda_{\max}\left(\left(L_{B^{c}}+\delta L_{B^{d}}\right)
\left(L_{B^{c}}+\delta L_{B^{d}}\right)^{\tp}\right)}{\Ds2(\delta-1)}
\end{equation}

\noindent where $H$ is determined with (\ref{Eq22}). By Corollary \ref{cor2}, $L_{B^{c}}+\delta L_{B^{d}}$ is positive stable for any $\delta>0$ when the union of $\mathscr{G}\left(B^{c}\right)$ and $\mathscr{G}\left(B^{d}\right)$ is strongly connected. We can hence obtain a positive definite matrix: $\left(L_{B^{c}}+\delta L_{B^{d}}\right)\left(L_{B^{c}}+\delta L_{B^{d}}\right)^{\tp}$. This, together with the positive definiteness of $H$, guarantees $\zeta>0$ for any $\delta>1$ under the topology conditions of Lemma \ref{lem5}. In particular, we may choose specific values of $\delta$ to determine $\zeta$ with (\ref{Eq41}).

With the above development, we can establish the following theorem to achieve the convergence analysis of directed signed networks subject to sign-inconsistent nonidentical topologies.

\begin{thm}\label{thm2}
Consider the system (\ref{eq03}) described by any sign-inconsistent signed digraphs $\mathscr{G}\left(B^{c}\right)$ and $\mathscr{G}\left(B^{d}\right)$, and let $k>\zeta$ be chosen for any $\delta>1$. If the union of $\mathscr{G}\left(B^{c}\right)$ and $\mathscr{G}\left(B^{d}\right)$ is strongly connected, then neutralization can be achieved such that for any $x(0)\in\mathbb{R}^{n}$ and $y(0)\in\mathbb{R}^{n}$, $\lim_{t\to\infty}x(t)=0$ and $\lim_{t\to\infty}y(t)=0$ hold.
\end{thm}

\begin{proof}
This theorem can be developed by resorting to the Lypunov-based analysis method in a similar way as Theorem \ref{thm1}, of which the proof details are given in the Appendix B.
\end{proof}

\begin{rem}\label{rem6}
In Theorem \ref{thm2}, we see that the sign-inconsistency of nonidentical topologies enables signed networks always to be neutralized. The neutralization analysis of Theorem \ref{thm2} also implies that the sign-inconsistency of nonidentical topologies may cause difficulties in implementing the behavior analysis of signed networks. It is mainly because for any union signed digraph $\mathscr{G}\left(B^{c}\right)\cup\mathscr{G}\left(B^{d}\right)$, the sign-inconsistency of $\mathscr{G}\left(B^{c}\right)$ and $\mathscr{G}\left(B^{d}\right)$ generally yields $\mathscr{G}\left(B^{c}\right)\cup\mathscr{G}\left(B^{d}\right)\neq\mathscr{G}\left(kB^{c}+\delta B^{d}\right)$, $\forall k>0$, $\forall\delta>0$. In fact, $\mathscr{G}\left(kB^{c}+\delta B^{d}\right)$ is generally only a subgraph of $\mathscr{G}\left(B^{c}\right)\cup\mathscr{G}\left(B^{d}\right)$. As a consequence, the stability property (or eigenvalue distribution) of $kL_{B^{c}}+\delta L_{B^{d}}$, $\forall k>0$, $\forall\delta>0$ can not be easily established in spite of the connectivity property of $\mathscr{G}\left(B^{c}\right)\cup\mathscr{G}\left(B^{d}\right)$. Fortunately, we can address this issue by an $M$-matrix-based analysis approach, which however has not been discussed in the literature of signed networks with nonidentical topologies (see, e.g., \cite{lm:18,mndz:18,nmd:16}).
\end{rem}

With Theorems \ref{thm1} and \ref{thm2}, we successfully obtain convergence results and analysis approaches for signed networks subject to nonidentical topologies, regardless of any sign pattern between signed digraphs representing the communication topologies of agents. We simultaneously explore a Lyapunov-based analysis approach to exploiting dynamic behaviors for signed networks in spite of the general directed topologies of agents, for which we also leverage an $M$-matrix approach.

%%%%%%%%%%%%%%%%%%%%%%%%%%%%%%%%%%%%%%%%%%%%%%%%%%%%%%%%%%%%%%%%%%%%%%%%%%%%%%%%%%%%%%%%%%%%%%%%%%%%%
\section{Conclusions}\label{sec4}

In this paper, convergence behaviors have been discussed for directed signed networks, on which the effects of nonidentical topologies have been investigated. A class of sign-consistency properties for any pairs of signed digraphs has been developed. It has been disclosed that the convergence behaviors of signed networks in presence of nonidentical topologies are associated closely with the sign-consistency property. More specifically, if the two signed digraphs describing the nonidentical topologies of signed networks are sign-consistent, then the states of all the agents polarize if and only if the union of two signed digraphs is structurally balanced; and neutralization emerges, otherwise. However, for signed networks with sign-inconsistent nonidentical topologies, they always become neutralized. Furthermore, a Lyapunov approach together with an $M$-matrix approach has been established for the behavior analysis of signed networks, which may be of independent interest in handling cooperative-antagonistic interactions over directed nonidentical topologies.

%%%%%%%%%%%%%%%%%%%%%%%%%%%%%%%%%%%%%%%%%%%%%%%%%%%%%%%%%%%%%%%%%%%%%%%%%%%%%%%%%%%
\section*{Acknowledgement}

The authors would like to thank the editors and reviewers for their suggestions and comments that help improve our paper.

%%%%%%%%%%%%%%%%%%%%%%%%%%%%%%%%%%%%%%%%%%%%%%%%%%%%%%%%%%%%%%%%%%%%%%%%%%%%%%%%%%%%%%%%%%%%%%%%%%%%%
\bibliographystyle{ieeetr}

%%%%%%%%%%%%%%%%%%%%%%%%%%%%%%%%%%%%%%%%%%%%%%%%%%%%%%%%%%%%%%%%%%%%%%%%%%%%%%%%%%%%%%%%
\section*{Appendix A: Proof of Theorem \ref{thm1}}

\begin{proof}[Proof of Theorem \ref{thm1}]
Since $\mathscr{G}\left(B^{c}\right)$ and $\mathscr{G}\left(B^{d}\right)$ are sign-consistent, we use Lemma \ref{lem2} to denote the union $\mathscr{G}\left(B\right)$ of them with $B=B^{c}+B^{d}$. Thus, the necessity results of this theorem follow directly by the mutually exclusive relationship between the structural balance and unbalance of $\mathscr{G}\left(B\right)$. To establish the sufficiency results, we also resort to the structural balance and unbalance of $\mathscr{G}\left(B\right)$, and consider two cases separately.

{\it Case i): $\mathscr{G}\left(B\right)$ is structurally balanced.} We have $DL_{B}D=L_{\left|B\right|}$ (or equivalently, $DBD=\left|B\right|$) for some $D\in\mathcal{D}_{n}$. Based on the sign-consistency of $\mathscr{G}\left(B^{c}\right)$ and $\mathscr{G}\left(B^{d}\right)$, we can deduce $DL_{B^{c}}D=L_{\left|B^{c}\right|}$ and $DL_{B^{d}}D=L_{\left|B^{d}\right|}$. Denote $Q=\left[1_{n}~F\right]$, and then $Q$ is invertible such that
\[
Q^{-1}=\begin{bmatrix}C\\E\end{bmatrix}~\hbox{with}~
C=\begin{bmatrix}1&0&\cdots&0\end{bmatrix}\in\mathbb{R}^{1\times n}.
\]

\noindent Let $\widetilde{x}(t)=Q^{-1}Dx(t)$ and $\widetilde{y}(t)=Q^{-1}Dy(t)$, with which we denote $\widetilde{x}(t)=\left[\widetilde{x}_{1}(t),\widetilde{x}_{2}^{\tp}(t)\right]^{\tp}$ and $\widetilde{y}(t)=\left[\widetilde{y}_{1}(t),\widetilde{y}_{2}^{\tp}(t)\right]^{\tp}$ for $\widetilde{x}_{1}(t)\in\mathbb{R}$, $\widetilde{x}_{2}(t)\in\mathbb{R}^{n-1}$, $\widetilde{y}_{1}(t)\in\mathbb{R}$ and $\widetilde{y}_{2}(t)\in\mathbb{R}^{n-1}$. From (\ref{eq03}) and with $L_{\left|B^{c}\right|}1_{n}=0$ and $L_{\left|B^{d}\right|}1_{n}=0$, we can validate
\begin{equation}\label{Eq34}
\aligned
\begin{bmatrix}
\dot{\widetilde{x}}(t)\\
\dot{\widetilde{y}}(t)
\end{bmatrix}
&=\begin{bmatrix}
-Q^{-1}DL_{B^{c}}DQ&I\\
-kQ^{-1}DL_{B^{d}}DQ&-kI
\end{bmatrix}
\begin{bmatrix}
\widetilde{x}(t)\\
\widetilde{y}(t)
\end{bmatrix}\\
&=\begin{bmatrix}
-Q^{-1}L_{\left|B^{c}\right|}Q&I\\
-kQ^{-1}L_{\left|B^{d}\right|}Q&-kI
\end{bmatrix}
\begin{bmatrix}
\widetilde{x}(t)\\
\widetilde{y}(t)
\end{bmatrix}\\
&=\begin{bmatrix}
-\begin{bmatrix}0&CL_{\left|B^{c}\right|}F\\0&EL_{\left|B^{c}\right|}F\end{bmatrix}
&I\\
-k\begin{bmatrix}0&CL_{\left|B^{d}\right|}F\\0&EL_{\left|B^{d}\right|}F\end{bmatrix}
&-kI
\end{bmatrix}
\begin{bmatrix}
\widetilde{x}(t)\\
\widetilde{y}(t)
\end{bmatrix}
\endaligned
\end{equation}

\noindent We can further explore (\ref{Eq34}) to obtain two subsystems as
\begin{equation}\label{Eq35}
\aligned
\begin{bmatrix}
\dot{\widetilde{x}}_{1}(t)\\
\dot{\widetilde{y}}_{1}(t)
\end{bmatrix}
&=\begin{bmatrix}
0&1\\
0&-k
\end{bmatrix}
\begin{bmatrix}
\widetilde{x}_{1}(t)\\
\widetilde{y}_{1}(t)
\end{bmatrix}
+\begin{bmatrix}
-CL_{\left|B^{c}\right|}F&0\\
-kCL_{\left|B^{d}\right|}F&0
\end{bmatrix}
\begin{bmatrix}
\widetilde{x}_{2}(t)\\
\widetilde{y}_{2}(t)
\end{bmatrix}
\endaligned
\end{equation}

\noindent and
\begin{equation}\label{Eq36}
\aligned
\begin{bmatrix}
\dot{\widetilde{x}}_{2}(t)\\
\dot{\widetilde{y}}_{2}(t)
\end{bmatrix}
&=\begin{bmatrix}
-EL_{\left|B^{c}\right|}F&I\\
-kEL_{\left|B^{d}\right|}F&-kI
\end{bmatrix}
\begin{bmatrix}
\widetilde{x}_{2}(t)\\
\widetilde{y}_{2}(t)
\end{bmatrix}.
\endaligned
\end{equation}

\noindent Clearly, the system performance of (\ref{Eq35}) relies on that of (\ref{Eq36}).

Next, we develop the convergence of the system (\ref{Eq36}). Since $\mathscr{G}(B)$ is strongly connected and structurally balanced, we use Lemma \ref{lem3} to give a Lyapunov function candidate for (\ref{Eq36}) as
\begin{equation}\label{Eq37}
V_{1}(t)=\begin{bmatrix}
\widetilde{x}_{2}(t)\\
\widetilde{y}_{2}(t)
\end{bmatrix}^{\tp}\begin{bmatrix}W&k^{-1}W\\k^{-1}W&k^{-2}\delta W\end{bmatrix}\begin{bmatrix}
\widetilde{x}_{2}(t)\\
\widetilde{y}_{2}(t)
\end{bmatrix}.
\end{equation}

\noindent Due to $k>0$ and $\delta>1$, $V_{1}(t)$ is positive definite based on the positive definiteness of $W$. Moreover, when we consider (\ref{Eq37}) for (\ref{Eq36}), we employ $L_{\left|B^{c}\right|}=DL_{B^{c}}D$ and $L_{\left|B^{d}\right|}=DL_{B^{d}}D$, utilize (\ref{Eq8}), and can arrive at a symmetric matrix given by
\begin{equation*}\label{}
\aligned
\widetilde{W}
=\begin{bmatrix}I&
k^{-1}\left[E\left(L_{\left|B^{c}\right|}+\delta L_{\left|B^{d}\right|}\right)F\right]^{\tp}W\\
(\star)&2k^{-1}(\delta-1)W\end{bmatrix}
\endaligned
\end{equation*}

\noindent where $(\star)$ represents the term induced by symmetry, such that
\begin{equation}\label{Eq38}
\dot{V}_{1}(t)
=-\begin{bmatrix}
\widetilde{x}_{2}(t)\\
\widetilde{y}_{2}(t)
\end{bmatrix}^{\tp}
\widetilde{W}
\begin{bmatrix}
\widetilde{x}_{2}(t)\\
\widetilde{y}_{2}(t)
\end{bmatrix}.
\end{equation}

\noindent With the Schur complement lemma (see, e.g., \cite[Theorem 7.7.6]{hj:85}), we can validate for (\ref{Eq38}) that the positive definiteness of $\widetilde{W}$ is equivalent to that of $\widehat{W}$, where
\begin{equation}\label{Eq39}
\aligned
\widehat{W}
&=\left(k^{-1}W\right)\Big\{2(\delta-1)kW^{-1}
-\left[E\left(L_{\left|B^{c}\right|}+\delta L_{\left|B^{d}\right|}\right)F\right]\\
&~~~~\times\left[E\left(L_{\left|B^{c}\right|}+\delta L_{\left|B^{d}\right|}\right)F\right]^{\tp}\Big\}
\left(k^{-1}W\right).
\endaligned
\end{equation}

\noindent Since both $\left[E\left(L_{\left|B^{c}\right|}+\delta L_{\left|B^{d}\right|}\right)F\right]
\left[E\left(L_{\left|B^{c}\right|}+\delta L_{\left|B^{d}\right|}\right)F\right]^{\tp}$ and $W^{-1}$ are positive definite from Lemma \ref{lem3} and Corollary \ref{cor1}, they are simultaneously diagonalizable based on \cite[Theorem 7.6.4]{hj:85}. By noting this fact and $k>\mu$, we can derive from (\ref{Eq39}) that $\widehat{W}$ is positive definite. Equivalently, $\widetilde{W}$ is positive definite. Thus, the use of (\ref{Eq38}) leads to the negative definiteness of $\dot{V}_{1}(t)$. From the standard Lyapunov stability theory, the system (\ref{Eq36}) is exponentially (or asymptotically) stable. Moreover, we employ the exponential stability of (\ref{Eq36}) and can use (\ref{Eq35}) to deduce
\begin{equation}\label{Eq40}
\aligned
\begin{bmatrix}
\dot{\widetilde{x}}_{1}(t)\\
\dot{\widetilde{y}}_{1}(t)
\end{bmatrix}
&=\begin{bmatrix}
\left(\aligned
&\widetilde{x}_{1}(0)-k^{-1}e^{-kt}\widetilde{y}_{1}(0)+\int_{0}^{t}\Big[-CL_{\left|B^{c}\right|}F\\
&+CL_{\left|B^{d}\right|}F e^{-k(t-\tau)}\Big]\widetilde{x}_{2}(\tau)d\tau
\endaligned\right)\\
e^{-kt}\widetilde{y}_{1}(0)
-kCL_{\left|B^{d}\right|}F
{\Ds\int_{0}^{t}}e^{-k(t-\tau)}\widetilde{x}_{2}(\tau)d\tau
\end{bmatrix}\\
&\to\begin{bmatrix}\widetilde{x}_{1\ast}\\0\end{bmatrix}
~\hbox{exponentially fast as}~t\to\infty
\endaligned
\end{equation}

\noindent where $\widetilde{x}_{1\ast}$ is a finite scalar given by
\[
\widetilde{x}_{1\ast}
=\widetilde{x}_{1}(0)
-CL_{\left|B^{c}\right|}F\int_{0}^{\infty}\widetilde{x}_{2}(\tau)d\tau.
\]

\noindent From the convergence result in (\ref{Eq40}) and the exponential stability of the system (\ref{Eq36}), we can directly derive the convergence of the state of the system (\ref{Eq34}). This also ensures the convergence of $X(t)$ for (\ref{eq03}) and $Y(t)$ for (\ref{eq04})
owing to the nonsingular linear transformation relationships as
\[
X(t)=\begin{bmatrix}DQ&0\\0&DQ\end{bmatrix}\begin{bmatrix}\widetilde{x}(t)\\\widetilde{y}(t)\end{bmatrix},
Y(t)=\Theta\begin{bmatrix}DQ&0\\0&DQ\end{bmatrix}\begin{bmatrix}\widetilde{x}(t)\\\widetilde{y}(t)\end{bmatrix}.
\]

To proceed, we calculate the converged value of $X(t)$. From two subsystems (\ref{Eq35}) and (\ref{Eq36}) separated from the system (\ref{Eq34}), it follows that there exists exactly one zero eigenvalue for the state matrix of (\ref{Eq34}), and the other eigenvalues have the positive real parts. With the algebraic equivalence between (\ref{eq04}) and (\ref{Eq34}), the same eigenvalue distribution applies to the state matrix $\Gamma$ of (\ref{eq04}). In the same way as \cite[eq. (15)]{mdj:16}, we can deduce
\begin{equation}\label{Eq18}
\lim_{t\to\infty}e^{\Gamma t}=w_{r}w_{l}^{\tp}\in\mathbb{R}^{2n\times2n}
\end{equation}

\noindent where $w_{l}\in\mathbb{R}^{2n}$ and $w_{r}\in\mathbb{R}^{2n}$ are eigenvectors for the zero eigenvalue of $\Gamma$ such that $\Gamma w_{r}=0$, $w_{l}^{\tp}\Gamma=0$ and $w_{l}^{\tp}w_{r}=1$. By the strong connectivity and structural balance of $\mathscr{G}\left(B\right)$ and the structure of $\Gamma$, we have the candidates of $w_{l}$ and $w_{r}$ as
\begin{equation}\label{Eq19}
w_{l}=\begin{bmatrix}D\nu\\0\end{bmatrix},
~~~w_{r}=\begin{bmatrix}D1_{n}\\k^{-1}D1_{n}\end{bmatrix}.
\end{equation}

\noindent The substitution of (\ref{Eq18}) and (\ref{Eq19}) into (\ref{eq04}) yields
\begin{equation*}%\label{Eq20}
\lim_{t\to\infty}Y(t)
=\begin{bmatrix}D1_{n}\\k^{-1}D1_{n}\end{bmatrix}\begin{bmatrix}D\nu\\0\end{bmatrix}^{\tp}Y(0)
\end{equation*}

\noindent which, together with $X(t)=\Theta^{-1}Y(t)$, leads to
\begin{equation}\label{Eq20}
\aligned
\lim_{t\to\infty}X(t)
&=\begin{bmatrix}0&I\\I&-kI\end{bmatrix}\lim_{t\to\infty}Y(t)\\
&=\begin{bmatrix}\left\{\nu^{\tp}D\left[x(0)+k^{-1}y(0)\right]\right\}D1_{n}\\0\end{bmatrix}.
\endaligned
\end{equation}

\noindent The converged value of (\ref{eq03}) in 1) of Theorem \ref{thm1} holds by (\ref{Eq20}).

{\it Case ii): $\mathscr{G}\left(B\right)$ is structurally unbalanced.} In this case, we apply Lemma \ref{lem3} to define a Lyapunov function candidate as
\begin{equation}\label{Eq23}
V_{2}(t)=X^{\tp}(t)\begin{bmatrix}H&k^{-1}H\\k^{-1}H&k^{-2}\delta H\end{bmatrix}X(t).
\end{equation}

\noindent We can deduce from (\ref{Eq23}) that $V_{2}(t)$ is positive definite because $k>\mu>0$, $\delta>1$ and $H$ is positive definite. Furthermore, by considering (\ref{Eq23}) for (\ref{eq03}), we can validate
\begin{equation}\label{Eq25}
\aligned
\dot{V}_{2}(t)
=-X^{\tp}(t)\widetilde{H}X(t)
\endaligned
\end{equation}

\noindent where, due to (\ref{Eq22}), $\widetilde{H}$ is given by
\begin{equation}\label{Eq26}
\widetilde{H}
=\begin{bmatrix}
I&
k^{-1}\left(L_{B^{c}}+\delta L_{B^{d}}\right)^{\tp}H\\
(\star)&2k^{-1}(\delta-1)H
\end{bmatrix}.
\end{equation}

\noindent In the same way as used in proving the positive definiteness of $\widetilde{W}$ in (\ref{Eq38}), we can verify from (\ref{Eq26}) that $\widetilde{H}$ is positive definite, and as a consequence of (\ref{Eq25}), $\dot{V}_{2}(t)$ is negative definite. Based on the standard Lyapunov stability theory, we can deduce that the system (\ref{eq03}) is exponentially stable such that $\lim_{t\to\infty}x(t)=0$ and $\lim_{t\to\infty}y(t)=0$.

With the analyses in Cases i) and ii), we gain the sufficiency proofs for the results 1) and 2) of Theorem \ref{thm1}, respectively.
\end{proof}

%%%%%%%%%%%%%%%%%%%%%%%%%%%%%%%%%%%%%%%%%%%%%%%%%%%%%%%%%%%%%%%%%%%%%%%%%%%%%%%%%%%%%%%%%%%%%%%%%%%%%
\section*{Appendix B: Proof of Theorem \ref{thm2}}

To prove Theorem \ref{thm2}, we need to first present the proofs of Lemmas \ref{lem6} and \ref{lem5}.

\begin{proof}[Proof of Lemma \ref{lem6}]
Because $\mathscr{G}\left(B^{c}\right)$ and $\mathscr{G}\left(B^{d}\right)$ are sign-inconsistent, there exist some pairs of $\{i,j\}$ such that $b_{ij}^{c}b_{ij}^{d}<0$. It thus yields $b_{ij}^{c-}+b_{ij}^{d-}<0$. This ensures $B^{c-}+B^{d-}\neq0$, based on which we can construct a nonnegative matrix as
\begin{equation}\label{Eq29}
\overline{A}=
\begin{bmatrix}
0&0\\
\Delta_{\left|B^{c-}+B^{d-}\right|}1_{n}
&B^{c+}+B^{d+}
\end{bmatrix}\geq0.
\end{equation}

\noindent Denote $\overline{A}\triangleq\left[\overline{a}_{ij}\right]\in\mathbb{R}^{(n+1)\times(n+1)}$. We can define an unsigned digraph $\overline{\mathscr{G}}\left(\overline{A}\right)=\left(\overline{\mathcal{V}},\overline{\mathcal{E}},\overline{A}\right)$, in which we set $\overline{\mathcal{V}}=\left\{v_{0}\right\}\cup\mathcal{V}\triangleq\left\{v_{i}:0\leq i\leq n\right\}$, and then let $\overline{\mathcal{E}}\subseteq\left\{\left(v_{j},v_{i}\right):\forall v_{i},v_{j}\in\overline{\mathcal{V}}\right\}$ be such that  $\left(v_{j},v_{i}\right)\in\overline{\mathcal{E}}\Leftrightarrow\overline{a}_{i+1,j+1}>0$ and $\left(v_{j},v_{i}\right)\not\in\overline{\mathcal{E}}\Leftrightarrow\overline{a}_{i+1,j+1}=0$. From (\ref{Eq29}), we can derive that $\overline{\mathscr{G}}\left(\overline{A}\right)$ has a path from $v_{0}$ to $v_{i}$, $\forall i\in\mathcal{I}_{n}$ if and only if $\sum_{j=1}^{n}\left(\left|b_{ij}^{c-}\right|+\left|b_{ij}^{d-}\right|\right)>0$. With this property, we first prove that the unsigned digraph $\overline{\mathscr{G}}\left(\overline{A}\right)$ contains a spanning tree, and then that $L_{B^{c+}+B^{d+}}+\Delta_{\left|B^{c-}+B^{d-}\right|}$ is an $M$-matrix to complete this proof.

We know from (\ref{Eq29}) that $\overline{\mathscr{G}}\left(\overline{A}\right)$ is composed of the unsigned digraph $\mathscr{G}\left(B^{c+}+B^{d+}\right)$ and the vertex $v_{0}$ with related arcs of $\left\{\left(v_{0},v_{i}\right):\sum_{j=1}^{n}\left(\left|b_{ij}^{c-}\right|+\left|b_{ij}^{d-}\right|\right)>0,i\in\mathcal{I}_{n}\right\}$. By the sign-inconsistency of $\mathscr{G}\left(B^{c}\right)$ and $\mathscr{G}\left(B^{d}\right)$, we consider $b_{l_{0}j_{0}}^{c}b_{l_{0}j_{0}}^{d}<0$ for some $l_{0}\in\mathcal{I}_{n}$ and $j_{0}\in\mathcal{I}_{n}$, with which it hence leads to $b_{l_{0}j_{0}}^{c-}+b_{l_{0}j_{0}}^{d-}<0$, and consequently,
\begin{equation*}%\label{Eq30}
\aligned
\sum_{j=1}^{n}\left(\left|b_{l_{0}j}^{c-}\right|+\left|b_{l_{0}j}^{d-}\right|\right)
&\geq\left|b_{l_{0}j_{0}}^{c-}\right|+\left|b_{l_{0}j_{0}}^{d-}\right|\\
&=-\left(b_{l_{0}j_{0}}^{c-}+b_{l_{0}j_{0}}^{d-}\right)>0.
\endaligned
\end{equation*}

\noindent Namely, there exists a directed path from $v_{0}$ to $v_{l_{0}}$ in $\overline{\mathscr{G}}\left(\overline{A}\right)$.

From the strong connectivity of $\mathscr{G}\left(B^{c}\right)\cup\mathscr{G}\left(B^{d}\right)$, it follows that for any $v_{i}$, $i\in\mathcal{I}_{n}$, this union has a path from $v_{l_{0}}$ to $v_{i}$, i.e., it has sequential arcs $\left(v_{l_{0}},v_{l_{1}}\right)$, $\left(v_{l_{1}},v_{l_{2}}\right)$, $\cdots$, $\left(v_{l_{m-1}},v_{l_{m}}\right)$ for distinct agents $v_{l_{0}}$, $v_{l_{1}}$, $\cdots$, $v_{l_{m}}$ (with $l_{m}=i$). Further, either $b_{l_{j}l_{j-1}}^{c}\neq0$ or $b_{l_{j}l_{j-1}}^{d}\neq0$, $\forall1\leq j\leq m$ holds. Equivalently, we can obtain either $b_{l_{j}l_{j-1}}^{c-}+b_{l_{j}l_{j-1}}^{d-}<0$ or  $b_{l_{j}l_{j-1}}^{c+}+b_{l_{j}l_{j-1}}^{d+}>0$, $\forall1\leq j\leq m$. If $b_{l_{m}l_{m-1}}^{c-}+b_{l_{m}l_{m-1}}^{d-}<0$, i.e., $b_{il_{m-1}}^{c-}+b_{il_{m-1}}^{d-}<0$, then we have
\begin{equation*}%\label{Eq31}
\aligned
\sum_{j=1}^{n}\left(\left|b_{ij}^{c-}\right|+\left|b_{ij}^{d-}\right|\right)
&\geq\left|b_{il_{m-1}}^{c-}\right|+\left|b_{il_{m-1}}^{d-}\right|\\
&=-\left(b_{il_{m-1}}^{c-}+b_{il_{m-1}}^{d-}\right)>0
\endaligned
\end{equation*}

\noindent which means that a directed path from $v_{0}$ to $v_{i}$ exists in $\overline{\mathscr{G}}\left(\overline{A}\right)$. Otherwise, if $b_{l_{m}l_{m-1}}^{c-}+b_{l_{m}l_{m-1}}^{d-}<0$ does not hold, we know $b_{l_{m}l_{m-1}}^{c+}+b_{l_{m}l_{m-1}}^{d+}>0$. By this fact, let $\widetilde{m}$ ($1\leq\widetilde{m}\leq m$) be the greatest integer that satisfies $b_{l_{j}l_{j-1}}^{c+}+b_{l_{j}l_{j-1}}^{d+}>0$, $\forall\widetilde{m}\leq j\leq m$ and $b_{l_{\widetilde{m}-1}l_{\widetilde{m}-2}}^{c-}+b_{l_{\widetilde{m}-1}l_{\widetilde{m}-2}}^{d-}<0$, where in particular if $\widetilde{m}=1$, we denote $l_{\widetilde{m}-2}$ as the integer $j_{0}$ fulfilling $b_{l_{0}j_{0}}^{c-}+b_{l_{0}j_{0}}^{d-}<0$. Thus, the use of $b_{l_{\widetilde{m}-1}l_{\widetilde{m}-2}}^{c-}+b_{l_{\widetilde{m}-1}l_{\widetilde{m}-2}}^{d-}<0$ leads to
\begin{equation*}%\label{Eq32}
\aligned
\sum_{j=1}^{n}\left(\left|b_{l_{\widetilde{m}-1}j}^{c-}\right|+\left|b_{l_{\widetilde{m}-1}j}^{d-}\right|\right)
&\geq\left|b_{l_{\widetilde{m}-1}l_{\widetilde{m}-2}}^{c-}\right|+\left|b_{l_{\widetilde{m}-1}l_{\widetilde{m}-2}}^{d-}\right|\\
&=-\left(b_{l_{\widetilde{m}-1}l_{\widetilde{m}-2}}^{c-}+b_{l_{\widetilde{m}-1}l_{\widetilde{m}-2}}^{d-}\right)\\
&>0
\endaligned
\end{equation*}

\noindent i.e., $\overline{\mathscr{G}}\left(\overline{A}\right)$ has a directed path from $v_{0}$ to $v_{l_{\widetilde{m}-1}}$. By resorting to $b_{l_{j}l_{j-1}}^{c+}+b_{l_{j}l_{j-1}}^{d+}>0$, $\forall\widetilde{m}\leq j\leq m$, we can obtain that $\left(v_{l_{j-1}},v_{l_{j}}\right)$, $\forall\widetilde{m}\leq j\leq m$ is an arc in $\mathscr{G}\left(B^{c+}+B^{d+}\right)$, and consequently,  $\left(v_{l_{j-1}},v_{l_{j}}\right)\in\overline{\mathcal{E}}$, $\forall\widetilde{m}\leq j\leq m$ holds. Based on these two facts, we can conclude that a path from $v_{0}$ to $v_{i}$ in $\overline{\mathscr{G}}\left(\overline{A}\right)$ is represented by the sequential arcs $\left(v_{0},v_{l_{\widetilde{m}-1}}\right)$, $\left(v_{l_{\widetilde{m}-1}},v_{l_{\widetilde{m}}}\right)$, $\cdots$, $\left(v_{l_{m-1}},v_{l_{m}}\right)$.

Summarising, we can deduce that $\overline{\mathscr{G}}\left(\overline{A}\right)$ admits paths from $v_{0}$ to every other vertex $v_{i}$, $\forall i\in\mathcal{I}_{n}$. In other words, $\overline{\mathscr{G}}\left(\overline{A}\right)$ has a spanning tree.

From (\ref{Eq29}), the Laplacian matrix of $\overline{\mathscr{G}}\left(\overline{A}\right)$ is given by
\begin{equation}\label{Eq30}
L_{\overline{A}}=
\begin{bmatrix}
0&0\\
-\Delta_{\left|B^{c-}+B^{d-}\right|}1_{n}
&L_{B^{c+}+B^{d+}}+\Delta_{\left|B^{c-}+B^{d-}\right|}
\end{bmatrix}.
\end{equation}

\noindent Since $\overline{\mathscr{G}}\left(\overline{A}\right)$ has a spanning tree, we can develop that for $L_{\overline{A}}$, there exists only one zero eigenvalue and the other eigenvalues of it are all with positive real parts (see, e.g., \cite[Lemma 3.3]{rb:05}). This together with (\ref{Eq30}) yields that $L_{B^{c+}+B^{d+}}+\Delta_{\left|B^{c-}+B^{d-}\right|}$ is positive stable. Due to also $L_{B^{c+}+B^{d+}}+\Delta_{\left|B^{c-}+B^{d-}\right|}\in\mathcal{Z}_{n}$ based on (\ref{Eq43}), it is immediately an $M$-matrix according to \cite[Definition 2.5.2]{hj:91}.
\end{proof}

\begin{proof}[Proof of Lemma \ref{lem5}]
By Lemma \ref{lem6}, we know $L_{B^{c+}+B^{d+}}+\Delta_{\left|B^{c-}+B^{d-}\right|}$ is an $M$-matrix. As a consequence, we can derive $\left(L_{B^{c+}+B^{d+}}+\Delta_{\left|B^{c-}+B^{d-}\right|}\right)^{-1}\geq0$. Thus, if we denote
\[
\Xi=I-\left(L_{B^{c+}+B^{d+}}+\Delta_{\left|B^{c-}+B^{d-}\right|}\right)^{-1}
\left(B^{c-}+B^{d-}\right)
\]

\noindent then with (\ref{Eq28}), $\Xi\geq0$ holds according to \cite[eqs. (8.1.5) and (8.1.11)]{hj:85}. From the sign-inconsistency of $\mathscr{G}\left(B^{c}\right)$ and $\mathscr{G}\left(B^{d}\right)$, it follows that $B^{c-}+B^{d-}\neq0$, and that $\mathscr{G}\left(B^{c}\right)$ and $\mathscr{G}\left(B^{d}\right)$ can not be simultaneously structurally balanced (namely, there does not exist any $D\in\mathcal{D}_{n}$ to ensure both $DB^{c}D=\left|B^{c}\right|$ and $DB^{d}D=\left|B^{d}\right|$). Moreover, by following the same lines as the proofs of \cite[Lemma 4.2 and Corollary 4.2]{dm:18}, we can develop that $\Xi$ is positive stable.

By leveraging the strong connectivity of $\mathscr{G}\left(B^{c}\right)\cup\mathscr{G}\left(B^{d}\right)$, we can deduce that $L_{B^{c}}+L_{B^{d}}$ is diagonally dominant and its diagonal entries are all positive. Based on the Ger\v{s}hgorin circle theorem (see, e.g., \cite[Theorem 6.1.1]{hj:85}), we can validate that the eigenvalues of $L_{B^{c}}+L_{B^{d}}$ either have positive real parts or are zero. In addition, we can employ (\ref{Eq42}) to derive
\begin{equation}\label{Eq31}
\aligned
\det\left(L_{B^{c}}+L_{B^{d}}\right)
&=\det\left(L_{B^{c+}+B^{d+}}
+\Delta_{\left|B^{c-}+B^{d-}\right|}\right)\det\left(\Xi\right)\\
&>0
\endaligned
\end{equation}

\noindent where the positive stability of both $L_{B^{c+}+B^{d+}}+\Delta_{\left|B^{c-}+B^{d-}\right|}$ and $\Xi$ is used. By $\det\left(L_{B^{c}}+L_{B^{d}}\right)
=\prod_{i=1}^{n}\lambda_{i}\left(L_{B^{c}}+L_{B^{d}}\right)$ (see, e.g., \cite[Theorem 1.2.12]{hj:85}), we can employ (\ref{Eq31}) to obtain that $L_{B^{c}}+L_{B^{d}}$ is positive stable.
\end{proof}

Now, we can present the proof of Theorem \ref{thm2} below.

\begin{proof}[Proof of Theorem \ref{thm2}]
In view of Lemma \ref{lem5} and Corollary \ref{cor2}, this proof can be developed based on the Lyapunov function candidate $V_{2}(t)$ in (\ref{Eq23}) by following the same steps as employed in the Case ii) of the proof of Theorem \ref{thm1} and, thus, is not detailed here.
\end{proof}

\end{document}